\documentclass{amsart}

\usepackage{amsfonts}
\usepackage{amsmath}
\usepackage{amsthm}
\usepackage{amssymb}
\usepackage[english]{babel}
\usepackage{graphics}
\usepackage{latexsym}
\usepackage{longtable} 
\usepackage{mathrsfs} 
\usepackage{graphicx}
\usepackage{wrapfig} 
\usepackage[pdftex,colorlinks=true,linkcolor=blue,urlcolor=blue,citecolor=blue]{hyperref}

\newtheorem{theorem}{Theorem}
\newtheorem{corollary}[theorem]{Corollary}
\newtheorem{lemma}[theorem]{Lemma}
\newtheorem{proposition}[theorem]{Proposition}

\newtheorem{claim}[theorem]{Claim}
\newtheorem{example}[theorem]{Example}

\theoremstyle{definition}
\newtheorem{definition}[theorem]{Definition}
\newtheorem{remark}[theorem]{Remark}

\newcommand{\Rco}{\operatorname*{Rco}} 
\newcommand{\Int}{\operatorname*{int}} 
\newcommand{\Pco}{\operatorname*{Pco}}
\newcommand{\Co}{\operatorname*{co}}

\newcommand{\mA}{\mathcal{A}}

\newcommand{\mF}{\mathcal{F}}

\newcommand{\mP}{\mathscr{P}}
\newcommand{\mM}{\mathcal{M}}

\newcommand{\mY}{\mathscr{Y}}
\newcommand{\mD}{\mathcal{D}}

\newcommand{\A}{\textbf{A}}
\newcommand{\R}{\mathbb{R}}

\newcommand{\N}{\mathbb{N}}

\newcommand{\X}{\textbf{X}}

\renewcommand{\H}{\mathrm{H}}

\newcommand{\noi}{\noindent}
\newcommand{\ms}{\medskip}

\newcommand{\al}{\alpha}
\newcommand{\be}{\beta}
\newcommand{\ga}{\gamma}

\newcommand{\de}{\delta}
\newcommand{\De}{\Delta}

\newcommand{\si}{\sigma}

\newcommand{\la}{\lambda}

\newcommand{\Om}{\Omega}

\newcommand{\D}{\mathrm{D}} 
 
\newcommand{\weak }{\, -\!\!\!\!-\!\!\!\rightharpoonup}
\newcommand{\weakstar }{ \overset{\, *_{\phantom{|}}}{{\smash{\weak }}\, } }

\newcommand{\larrow}{\longrightarrow}
\newcommand{\ot}{\otimes}

\newcommand{\lmapsto}{\longmapsto}
\newcommand{\ri}{\rightarrow}

\newcommand{\p}{\partial}
\newcommand{\sub}{\subseteq}

\newcommand{\by}{\times}
\newcommand{\rk}{\mathrm{rk}}

\newcommand{\ess}{\mathrm{ess}}

\newcommand{\dist}{\mathrm{dist}}

\newcommand{\Div}{\mathrm{Div}}

\newcommand{\supp}{\mathrm{supp}}

\newcommand{\bt}{\begin{theorem}}\newcommand{\et}{\end{theorem}}
\newcommand{\bd}{\begin{definition}}\newcommand{\ed}{\end{definition}}
\newcommand{\bl}{\begin{lemma}}\newcommand{\el}{\end{lemma}}
\newcommand{\beq}{\begin{equation}}\newcommand{\eeq}{\end{equation}}
\newcommand{\bc}{\begin{claim}}\newcommand{\ec}{\end{claim}}
\newcommand{\bex}{\begin{example}}\newcommand{\eex}{\end{example}}
\newcommand{\bcor}{\begin{corollary}}\newcommand{\ecor}{\end{corollary}}
\newcommand{\bp}{\begin{proof}}\newcommand{\ep}{\end{proof}}

\newcommand{\BPC}{\medskip \noindent \textbf{Proof of Claim} }

\newcommand{\BPP}{\medskip \noindent \textbf{Proof of Proposition} }
\newcommand{\BPT}{\medskip \noindent \textbf{Proof of Theorem} }

\numberwithin{equation}{section}

\begin{document}

\title[$\mD$-solutions in $L^\infty$ via the singular value problem]{$\mD$-solutions to the system of vectorial Calculus of Variations in $L^\infty$ via the singular value problem}

\author{Gisella Croce, Nikos Katzourakis and Giovanni Pisante}
\address{Normandie Univ, UNIHAVRE, LMAH, 76600 Le Havre, FranceÃÂ}
\email{gisella.croce@univ-lehavre.fr}

\address{Department of Mathematics and Statistics, University of Reading, Whiteknights, PO Box 220, Reading RG6 6AX, UK}
\email{n.katzourakis@reading.ac.uk}

  \thanks{\!\!\!\!\!\!\texttt{N.K. has been partially financially supported by the EPSRC grant EP/N017412/1}}

\address{Dipartimento di Matematica e Fisica, Seconda Universit\`a degli Studi di Napoli - S.U.N., Viale Lincoln, 5, c.a.p.  81100 Caserta, Italy}
\email{giovanni.pisante@unina2.it}



\date{}


\keywords{Vectorial Calculus of Variations in $L^\infty$; Generalised solutions;  Fully nonlinear systems; $\infty$-Laplacian; Young measures; Singular Value Problem, Baire Category method; Convex Integration}

\begin{abstract} {For $\mathrm{H}\in C^2(\R^{N\by n})$ and $u :\Omega \subseteq \R^n \to \R^N$, consider the system 
\[ \label{1}
\mathrm{A}_\infty  u\, :=\,\Big(\mathrm{H}_P \otimes \mathrm{H}_P + \mathrm{H}[\mathrm{H}_P]^\bot \mathrm{H}_{PP}\Big)(\D u):\D^2u\, =\,0. \tag{1}
\]
We construct $\mD$-solutions to the Dirichlet problem for {\eqref1}, an apt notion of generalised solutions recently proposed for fully nonlinear systems. Our $\mD$-solutions are $W^{1,\infty}$-submersions and are obtained without any convexity hypotheses for $\mathrm{H}$, through a result of independent interest involving existence of strong solutions to the singular value problem for general dimensions $n\neq N$.
}
\end{abstract} 

\maketitle


\section{Introduction} \label{section1}

Let $\H \in C^2(\R^{N\by n})$ be a given function and $\Om\sub \R^n$ a given open set, $n,N\in \N$. In this paper we are interested in the problem of existence of appropriately defined generalised solutions with given Dirichlet boundary conditions to the second order PDE system
\begin{equation}\label{1.1}
\mathrm{A}_\infty  u\, :=\,\Big(\H_P \ot \H_P + \H[\H_P]^\bot \H_{PP}\Big)(\D u):\D^2u\, =\,0. 
\end{equation}
In the above, the subscript $P$ denotes the derivative of $\H$ with respect to its matrix variable, while 
\[
\begin{split}
\D u(x)\ &=\ \big( \D_i u_\al(x)\big)_{i=1,...,n}^{\al=1,...,N} \ \ \ \, \in \R^{N\by n},
\\ 
\D^2 u(x)\ &=\ \big( \D^2_{ij} u_\al(x)\big)_{i,j=1,...,n}^{\al=1,...,N} \ \  \in \R^{N\by n^2}_s,
\end{split}
\]
denote respectively the gradient matrix and the hessian tensor of (smooth) maps $u : \Om\sub \R^n \larrow \R^N$. The notation ``$[\H_P]^\bot$'' symbolises the orthogonal projection on the orthogonal complement of the range of the linear map $\H_P(P) : \R^n \larrow \R^N$:
\begin{equation}\label{1.2}
[\H_P(P)]^\bot\,:=\ \textrm{Proj}_{R(\H_P(P))^\bot}.
\end{equation}
In index form, \eqref{1.1} reads
\[
\sum_{\be=1}^N\sum_{i,j=1}^n\left(\H_{P_{\al i}}(\D u)\,\H_{P_{\be j}}(\D u) \,+\, \H(\D u)\sum_{\ga=1}^N \big[\H_P(\D u)\big]_{\al \ga}^\bot \H_{P_{\ga i}P_{\be j}}(\D u)\right)\, \D_{ij}^2u_\be\, =0,
\]
$\al=1,...,N$. Our general notation is either self-explanatory or a convex combination of standard symbolisations as e.g.\ in \cite{E, D, EG, DM2}. The system \eqref{1.1} is the analogue of the Euler-Lagrange equation when one considers nonstandard vectorial variational problems in the space $L^\infty$ for the supremal functional
\begin{equation}\label{1.3}
\ \ \ \ \mathrm{E}_\infty(u,\Omega')\, :=\, \big\| \H(\D u)\big\|_{L^\infty(\Om')},
\ \ \ u\in W^{1,\infty}_{\text{loc}}(\Om,\R^N),\ \  \Omega'\Subset \Omega
\end{equation}
and first arose in recent work of the second author (\cite{K1}). Calculus of Variations in $L^\infty$, as the field is known today, was initiated by G.\ Aronsson in the 1960s who studied the scalar case $N=1$ quite systematically (\cite{A1}-\cite{A7}). Since then the area has been developed marvellously due to both the intrinsic mathematical interest and the importance for applications. In particular, the theory of Viscosity Solutions of Crandall-Ishii-Lions played a fundamental role in the study of the generally singular solutions to the scalar version of \eqref{1.1}. When $N=1$, the respective single equation simplifies to
\begin{equation}\label{1.4}
\mathrm{A}_\infty  u \, =\, \H_P(\D u)\ot \H_P(\D u) :\D^2u\,=\,0
\end{equation}
and is known as the ``Aronsson equation". For a pedagogical introduction to the scalar case with numerous references see e.g.\ \cite{K7,C} (see also \cite{ACJ, BEJ, SS, CR} for some relevant work and \cite{Pi} for a comparison between the Viscosity Solutions approach and the Baire Category one). 

On the other hand, except perhaps the deep vectorial contributions in \cite{BJW1, BJW2} which however were restricted to the study of the functional only, until the early 2010s the case $N\geq 2$ remained terra incognita. The foundations of the general vector case have been laid in a series of recent papers of one of the authors (\cite{K1}-\cite{K6}, \cite{K8}-\cite{K12}) and also in collaboration with Abugirda, Manfredi and Pryer (\cite{AK, KP, KM}). An important special case of \eqref{1.1} to which the results of this paper apply is the archetypal model of the so-called $\infty$-Laplace system
\begin{equation}\label{1.5}
\left\{\ \ 
\begin{split}
\De_\infty u\, :=&\,\Big(\D u \ot \D u + |\D u|^2[\D u]^\bot \ot \mathrm{I}\Big):\D^2u\, =\,0,
\\
&\ \ \ \ \ \ [\D u]^\bot \, =\, \textrm{Proj}_{R(\D u)^\bot},
\end{split}
\right.
\end{equation}
which arises by taking as $\H$ the square of the Euclidean norm of $\R^{N\by n}$:
\[
\ \ \ \ (u,\Omega')\, \lmapsto \, \big\| |\D u|^2\big\|_{L^\infty(\Om')},
\ \ \ u\in W^{1,\infty}_{\text{loc}}(\Om,\R^N),\ \  \Omega'\Subset \Omega.
\]
In index form, \eqref{1.5} becomes
\[
\sum_{\be=1}^N\sum_{i,j=1}^n\Big(\D_i u_\al \D_ju_\be \,+\, |\D u|^2[\D u]_{\al \be}^\bot \,\de_{ij}\Big)\, \D_{ij}^2u_\be\, =\,0, \ \ \ \ \al=1,...,N. 
\]
An additional difficulty in the study of the non-divergence system \eqref{1.1} which is not present in the scalar case of \eqref{1.4} is that the operator $\mathrm{A}_\infty $ \emph{has discontinuous coefficients, even when it is applied to $C^\infty$ maps}. Actually, even in the model case of \eqref{1.5}, the example $u(x_1,x_2)=e^{ix_1}-e^{ix_2}$ is a $2\by 2$ analytic $\infty$-Harmonic map near the origin of $\R^2$, but the rank of the gradient jumps from $2$ off the diagonal $\{x_1=x_2\}$ to $1$ on it and $[\D u]^\bot$ is discontinuous. The emergence of interfaces whereon the coefficients are discontinuous is a general phenomenon studied in some detail in \cite{K1,K3,K4}. A manifestation of this fact is that \emph{in the genuine vectorial case of rank $\rk(\D u)\geq 2$, the appropriate minimality notion connecting \eqref{1.1} to \eqref{1.3} is not the obvious extension of Aronsson's scalar notion of Absolute Minimisers} (see \cite{K2, K10, K11}).

Perhaps the greatest difficulty associated to the study of \eqref{1.1} is that it is quasilinear, non-divergence and non-monotone and all standard approaches in order to define generalised solutions based on integration-by-parts or on the maximum principle seem to fail. Motivated partly by the systems arising in Calculus of Variations in $L^\infty$, the second author has recently proposed in \cite{K9,K8} a new efficient theory of generalised solutions which applies to fully nonlinear PDE systems of any order
\[ 
\mF\Big(x,u(x),\D u(x),...,\D^pu(x)\Big)\, =\, 0, \quad x\in\Om,
\]
and allows for \emph{merely measurable} mappings to be admitted as solutions.  This approach of the so-called \emph{$\mD$-solutions} is duality-free and is based on the probabilistic representation of derivatives which do not exist classically. The tools is the weak* compactness of the difference quotients considered in the Young measures valued into a compactification of the ``space of jets".

For the special case at hand of $2$nd order systems like \eqref{1.1}-\eqref{1.5}, the definition can be motivated as follows: let $u$ be a strong a.e.\  $W^{2,\infty}_{\text{loc}}(\Om,\R^N)$ solution to
\begin{equation}  \label{1.6}
\ \ \ \mF\big(\D u,\D^2u\big)\,=\, 0, \quad \text{a.e.\ on }\Om.
\end{equation}
We are looking for a meaningly way to interpret the hessian rigorously for just $W^{1,\infty}_{\text{loc}}(\Om,\R^N)$ because this is the natural regularity class for \eqref{1.1} arising from the $L^\infty$ variational problem for \eqref{1.3}. For reasons to become apparent in a moment, we restate \eqref{1.6} as: for any compactly supported $\Phi \in C_c\big( \R^{N\by n^2}_s \big)$, we have
\begin{equation} \label{1.7}
\ \ \ \int_{\R^{N\by n^2}_s} \Phi(\X)\, \mF\big(\D u(x),\X\big)\, d[\de_{\D^2 u(x)} ](\X)\, =\, 0, \quad \text{ a.e. }x\in \Om.
\end{equation}
That is, we view the classical hessian map $\D^2 u$ as a probability-valued map $\Om\sub\R^n \larrow \mathscr{P}\big(\R^{N\by n^2}_s\big)$ given by the Dirac measure at itself $x\mapsto \de_{\D u(x)}$. Further, we restate that $\D^2 u$ is the limit in measure of the difference quotients of the gradient $\D^{1,h}\D u$ as $h\ri 0$ by writing
\begin{equation} \label{1.8}
\ \ \ \de_{\D^{1,h}\D u} \weakstar\, \de_{\D^2 u}, \quad \text{as }h\ri0.
\end{equation}
The weak* convergence of \eqref{1.8} is in the Young measures valued into $\R^{N\by n^2}_s$, that is the set of weakly* measurable probability-valued maps $\Om\sub\R^n \larrow \mathscr{P}\big(\R^{N\by n^2}_s\big)$ (for details see Section \ref{section2} that follows and \cite{CFV, FG, V, P, FL}). The hope arising from \eqref{1.7}-\eqref{1.8} is that it might be possible for general probability-valued ``diffuse hessians" to arise for just $W^{1,\infty}_{\text{loc}}$ maps which will no longer be the concentration masses $\de_{\D^2 u}$. This is indeed possible if we embed $\R^{N\by n^2}_s$ into its $1$-point Alexandroff compactification 
\[
\smash{\overline{\R}}^{N\by n^2}_s\, :=\, \R^{N\by n^2}_s\cup\{\infty\}. 
\]
By considering $(\de_{\D^{1,h}\D u})_{h\neq 0}$ as Young measures valued into $\smash{\overline{\R}}^{N \by n^2}_s$, subsequential weak* limits always do exist in the set of Young measures $\Om\sub\R^n \larrow \mathscr{P}\big(\smash{\overline{\R}}^{N \by n^2}_s\big)$ and we may define generalised hessians as follows:

\begin{definition}[Diffuse Hessians, cf.\ \cite{K8}]  \label{Diffuse Derivatives}
For any $ u \in W^{1,\infty}_{\text{loc}}(\Om,\R^N)$, we define the diffuse hessians $\mD^2 u$ as the infinitesimal subsequential limits of difference quotients of $\D u$ in the Young measures valued into the sphere $\smash{\overline{\R}}^{N\by n^2}_s$:
\[
\ \ \de_{\D^{1,h_{m}}\D u}\weakstar \mD^2 u, \  \ \ \, \text{ in }\mY\big(\Om,\smash{\overline{\R}}^{N\by n^2}_s\big), \ \ \ \text{ as }m\ri \infty.
\]
\end{definition} 
Obviously,
\[
\D^{1,h}v(x)\, =\,  \left(\frac{v(x+he^1)-v(x)}{h},...\, ,\frac{v(x+he^n)-v(x)}{h} \right),\ \ \ h\neq 0
\]
where $\{e^1,...,e^n\}$ is the usual basis of $\R^n$. By the weak* compactness of the set $\mY\big(\Om,\smash{\overline{\R}}^{N\by n^2}_s\big)$, every map possesses at least one diffuse hessian (and actually exactly one if it is twice differentiable in measure \`a la Ambrosio-Mal\'y \cite{AM}, see \cite{K8}).

\begin{definition}[Lipschitzian $\mD$-solutions of $2$nd order PDE systems, cf.\ \cite{K8}] \label{definition11} Let 
\[
\ \mF\ : \ \ \R^{N\by n}\!\by \R^{N \by n^2}_s \larrow \R^N
\]
be a Borel measurable map. A mapping  $u : \Om\sub \R^n \larrow \R^N$ in $W^{1,\infty}_{\text{loc}}(\Om,\R^N)$ is a $\mD$-solution to 
\begin{equation}\label{2.11}
\mF\big(\D u,\D^2 u\big)\, =\,0 \ \ \text{ in }\Om,
\end{equation}
if for any diffuse hessian $\mD ^2 u \in \mY(\Om,\smash{\overline{\R}}^{N\by n^2}_s)$ and any $\Phi \in C_c\big( {\R}^{N\by n^2}_s \big)$, we have
\[
\ \ \ \int_{\smash{\overline{\R}}^{N\by n^2}_s} \Phi(\X)\, \mF\big(\D u(x),\X \big)\, d\big[\mD^2 u(x) \big](\X)\, =\, 0, \ \ \text{ a.e.\ }x\in \Om.
\] 
 \end{definition}

The notion of generalised solution given by Definitions \ref{Diffuse Derivatives} \& \ref{definition11} will be our central notion of solution for \eqref{1.1}. By the motivation of the notion and the results of Section \ref{section2} it follows that it is compatible with strong and classical solutions (for more see \cite{K8}-\cite{K12}).
 {More precisely, if a $\mD$-solution is twice differentiable a.e.\  on $\Om$, then 
\eqref{1.8} says that all diffuse derivatives coincide a.e.\ on $\Om$ and are given by the Dirac measure $\de_{\D^2 u}$ at the hessian. Hence, the integral formula \eqref{1.7} reduces to the definition of strong a.e.\ solutions.}

The principal result of this paper is the existence of $\mD$-solutions to the Dirichlet problem for \eqref{1.1} when $N\leq n$. These solutions have extra geometric properties, being $W^{1,\infty}$-submersions and solving a certain system of vectorial Hamilton-Jacobi equations associated to \eqref{1.1}-\eqref{1.3}. Our \emph{only assumptions} on $H$ are that
\begin{itemize}

\item $\H$ is non-negative, $C^1$ on $\R^{N\by n}$ and $C^2$ on $\{P \in \R^{N\by n} : \rk(P)=N\}$, 

\item $\H(P)$ depends on $P$ via $PP^\top \in \R^{N\by N}_s$,

\item $\rk(\H_P(P))=N$ when $\rk(P)=N$.

\end{itemize}
In particular, we \emph{do not assume any kind of convexity of ``BJW-quasiconvexity" or level-convexity for $\H$, not even $C^2$ smoothness on the whole of $\R^{N\by n}$} but merely on the open set of matrices of full rank. The notion of ``BJW-quasiconvexity" introduced and studied in \cite{BJW1, BJW2} is the correct notion for weak* lower semi-continuity of $L^\infty$ functionals. Examples of $\H$ to which our results apply are given  by $\H_1(P)=|AP|^\al$ for $\al>1$ and $A\in \R^{N\by N}_{\geq 0}$ and by $\H_2(P)=(|P|^2+1)^{-1}$. The example $\H_2$ does not even have rank-one convex sublevel sets. The solutions we construct  are obtained by the celebrated \emph{Baire Category method} of Dacorogna-Marcellini (see \cite{DM2} and \cite{D, DP}) for a certain $1$st order differential inclusion relevant to \eqref{1.1} which we expound after the statement:

\begin{theorem}[Existence of $\mD$-solutions to the Dirichlet problem] \label{theorem3} Let $N\leq n$ and $\H : \R^{N\by n} \larrow [0,\infty)$ be such that 
\[
H(P)\, =\, h\big(PP^\top\big)
\]
for some $h : \R^{N\by N}_{\geq 0} \larrow [0,\infty)$ satisfying
\begin{itemize}

\item $h \in C^1\big(\R^{N\by N}_{\geq 0}\big) \cap C^2\big(\R^{N\by N}_{> 0}\big)$ ,

\item the derivative $h_X(X)$ is symmetric and $\det(h_X(X))\neq 0$ for $X>0$.

\end{itemize}
Then, for any open $\Om\sub \R^n$ and $g\in W^{1,\infty}(\Om,\R^N)$, the Dirichlet problem for \eqref{1.1}
\begin{equation}\label{1.10}
\left\{
\begin{array}{rl}
\mathrm{A}_\infty  u \, =\, 0, & \text{ in }\Om,\\
u\, =\, g,  & \text{ on }\p\Om,
 \end{array}
 \right.
\end{equation}
has an infinite set of $\mD$-solutions in the class of Lipschitz submersions 
\[
\mA\,:=\, \Big\{v\in W^{1,\infty}_g(\Om,\R^N)\ :\ \rk(\D v)=N, \text{ a.e.\ on }\Om \Big\}.
\]
Namely, there is an infinite set of $u\in \mA$ such that for any $\mD ^2 u \in \mY(\Om,\smash{\overline{\R}}^{N\by n^2}_s)$,
\[
\ \ \ \int_{\smash{\overline{\R}}^{N\by n^2}_s} \Phi(\X)\, \Big(\H_P \ot \H_P + \H[\H_P]^\bot \H_{PP}\Big)(\D u):\X\, d[\mD^2u](\X)\, =\, 0
\] 
a.e.\ on $\Om$, for any $\Phi \in C_c\big( {\R}^{N\by n^2}_s \big)$.
\end{theorem}

\begin{proposition}[Geometric properties of $\mD$-solutions to \eqref{1.10}]  \label{corollary4} In the setting of Theorem \ref{theorem3}, we also have the following extra properties for our $\mD$-solutions: for any 
\[
c\, >\, \|\D g\|_{L^\infty(\Om)}
\]
there is an infinite set of solutions $\mA_c\sub \mA$ such that each $u \in \mA_c$ solves the coupled system of Hamilton-Jacobi equations
\begin{equation}\label{1.11}
\left\{
\begin{array}{ll}
\H(\D u)=\H\big([c\,\mathrm{I}\,|\,\mathrm{0}\, ]\big), & \text{ a.e.\ on }\Om,\\
\big[\H_P(\D u)\big]^\bot\, =\,0, & \text{ a.e.\ on }\Om,\ms\\
\det\big(\D u \D u^{\!\top} \big)\, =\, c^{2N},  & \text{ a.e.\ on }\Om,\ms\\
u\,=\,g, & \text{ on }\p\Om.
 \end{array}
 \right.
\end{equation}
\end{proposition}
Theorem \ref{theorem3} and Proposition \ref{corollary4} generalise one of the results of \cite{K8} in which the the conclusions above were established in the special case $n=N$ and $\H(P)=|P|^2$, corresponding to the $\infty$-Laplacian \eqref{1.5}. 

Informally, the idea of the proofs of Theorem \ref{theorem3} and Proposition \ref{corollary4} is as follows. An inspection of \eqref{1.1} shows that it can be contracted as
\begin{equation}\label{1.12}
\H_P(\D u)\,\D\big(\H(\D u)\big) \, +\, \H(\D u)\big[\H_P(\D u)\big]^\bot \Div\big(\H_{P}(\D u)\big)\, =\,0,
\end{equation}
that is as
\[
 \sum_{i=1}^n \Bigg(\H_{P_{\al i}}(\D u)\, \D_i\big(\H(\D u)\big) \, +\, \H(\D u)\sum_{\ga=1}^N\big[\H_P(\D u)\big]_{\al \ga}^\bot \D_i\big(\H_{P_{\ga i}}(\D u)\big)\Bigg)\, =\,0,
\]
for $\al=1,...,N$. Hence, if we could prove for some $C\geq 0$ existence of solutions to the differential inclusion
\[
\D u(x)\, \in \, \Big\{P\in \R^{N\by n}\Big|\, \H(P)=C, \ \rk(\H_P(P))=N\Big\}, \ x\in \Om,
\]
we would obtain a solution to \eqref{1.12} (i.e.\ to \eqref{1.1}) because then $\D\big(\H(\D u)\big)\equiv 0$ and also $[\H_P(\D u)]^\bot$ since if $\H_P(\D u)$ has full rank the orthogonal complement of its range trivialises. However, the preceding arguments make sense only for classical or strong solutions. The proof of Theorem \ref{theorem3} has two main parts. We first use the Baire Category method to establish the existence of $W^{1,\infty}$ strong a.e.\ solutions to  system \eqref{1.11} of Corollary \ref{corollary4} and then we utilise the machinery of $\mD$-solutions to make the above heuristics rigorous.

Our main ingredient for the solvability of \eqref{1.11} is a result of independent interest about the solvability of the following fully non-linear system, usually referred to as the \emph{prescribed singular value problem}:
\begin{equation}\label{1.13}
\left\{
\begin{array}{rl}
\la_i(\D u)=1, & \text{ a.e.\ on }\Om, \, i=1,...,n\wedge N, \smallskip\\
u\,=\,g, & \text{ on }\p\Om.
 \end{array}
 \right.
\end{equation}
In \eqref{1.13}, $\{\la_1(\D u),...,\la_{n\wedge N}(\D u)\}$ denotes the set of singular values of the matrix $\D u$, namely the eigenvalues of the matrix $(\D u^{\!\top} \D u)^{1/2}$ in increasing order and we symbolise $n\wedge N = \min\{n,N\}$. 

Systems of PDEs involving singular values, mostly related to non-convex problems in Calculus of Variations, have been considered by several authors (cf.\ for instance \cite{Cr, BCR, DR, DPR}). In particular the problem of finding sufficient conditions on the boundary datum $g$ in order to get existence of solutions to problem \eqref{1.13} has been addressed, in the special case  $n=N$, in  \cite{DM1, DM2, DT}. To the best of our knowledge no results are known for the case $n\neq N$. Accordingly, we establish the following result.

\begin{theorem}\label{mainAff}
	\label{mainTH}
	Let $\Omega \sub \R^n$ be an open set. Assume that $g \in \text{\emph{Aff}}_{\text{\emph{pw}}}(\overline{\Omega}, \mathbb{R}^N)$ is such that $\la_{n\wedge N}(\D g)<1$, a.e.\ on $\Omega$. Then, there exists an infinite set of solutions $u\in W_g^{1,\infty}(\Omega,\mathbb{R}^N)$ to the system (\ref{1.13}).
	\end{theorem}
The proof of the previous theorem can be obtained as an application of the general existence theory for differential inclusion via the Baire Category method (cf.\ \cite{DP}), in the same spirit as in the $N=n$ case. It relies on the characterisation of the rank-one convex envelope of the set of matrices
\[
E\, := \, \Big\{Q \in \R^{N\times n}\, :\, \la_i(Q)=1,\ i=1,\dots, N\wedge n \Big\},
\]
which is
\[
\Rco E \,=\, \Big\{Q \in \R^{N\times n}\, : \, \lambda_{N\wedge n}(Q)\leq 1 \Big\},
\]
as proved in Theorem \ref{th:rco}.

In order to address the question of the existence of solutions to the problem \eqref{1.11} with $g\in W^{1,\infty}(\Omega,\mathbb{R}^N)$, some comments on the admissible regularity of the boundary datum in Theorem \ref{mainAff} are in order. Indeed, the piecewise affinity of the datum $g$ can be weakened to Lipschitz continuity if we restrict slightly the bound on the norm $\la_{n\wedge N}(\D g)$. This result, precisely stated in Corollary \ref{cor:main} that follows, is a simple consequence of the convexity of the rank-one convex hull of $E$ (cf.\ Theorem \ref{th:rco}) and of the approximation result proved in \cite[Corollary 10.21]{DM2}. 

\begin{corollary}
	\label{cor:main}
	Let $\Omega \sub \R^n$ be an open set. Assume that $g \in W^{1,\infty}(\Omega,\mathbb{R}^N)$ is such that for some $\delta>0$ we have $\la_{n\wedge N}(\D g)\leq 1-\delta $, a.e. on $\Omega$. Then there exists a infinite set of  solutions $u\in W_g^{1,\infty}(\Omega,\mathbb{R}^N)$ to the system (\ref{1.13}).
\end{corollary}

The rest of the paper is organised as follows. In the next section we recall some known results about Young measures valued into spheres. In Section \ref{section3} we provide the proof of the existence of solutions for the prescribed singular value problem and in the last section we prove existence and geometric properties of $\mathcal D$-solutions to the problem \eqref{1.10}.

\section{Young measures valued into spheres} \label{section2}

Here we collect some basic material taken from \cite{K8} which can be found in different guises and in greater generality e.g.\ in \cite{CFV, FG}. Let $E\sub \R^n$ be measurable and consider the $L^1$ space of strongly measurable maps valued in the continuous functions over the sphere $\smash{\overline{\R}}^{N\by n^2}_s$ (for details on these spaces see e.g.\ \cite{Ed, FL}):
\[
L^1\big( E, C\big(\smash{\overline{\R}}^{N\by n^2}_s\big)\big).
\]
The Banach space $C\big(\smash{\overline{\R}}^{N\by n^2}_s\big)$ will be endowed with the standard supremum norm. The above $L^1$ space consists of Carath\'eodory functions $\Phi : E \by \smash{\overline{\R}}^{N\by n^2}_s \larrow \R$  for which 
\[
\| \Phi \|_{L^1( E, C(\smash{\overline{\R}}^{N\by n^2}_s))}\, =\, \int_E \big\|\Phi(x,\cdot)\big\|_{C(\smash{\overline{\R}}^{N\by n^2}_s)}\, dx\, <\, \infty.
\]
The dual of this (separable) Banach space is
\[
L^\infty_{w^*}\big( E, \mM\big(\smash{\overline{\R}}^{N\by n^2}_s\big)\big)\, =\, L^1\big( E, \mM\big(\smash{\overline{\R}}^{N\by n^2}_s\big)\big)^* .
\]
The dual space above consists of measure-valued maps $x \mapsto \vartheta(x) $ which are weakly* measurable, that is, for any fixed $\Psi \in C\big(\smash{\overline{\R}}^{N\by n^2}_s\big)$, the function 
\[
E\ni\ x\lmapsto \int_{\smash{\overline{\R}}^{N\by n^2}_s}\Psi(\X)\, d[\vartheta(x)](\X) \ \in \R
\]
is measurable. The norm of the space is
\[
\| \vartheta \|_{L^\infty_{w^*} ( E,\mM (\smash{\overline{\R}}^{N\by n^2}_s) )}\, =\, \underset{x\in E}{\ess\,\sup}\, \left\|\vartheta(x) \right\|
\]
where ``$\|\cdot\|$" symbolises the total variation on the real (signed) Radon measures. The closed unit ball of $L^\infty_{w^*}$ is sequentially weakly* compact and the duality pairing
\begin{equation}\label{dp}
\langle\cdot,\cdot\rangle\ :\ \ \ L^\infty_{w^*}\big( E, \mM\big(\smash{\overline{\R}}^{N\by n^2}_s\big)\big) \by L^1\big( E, C\big(\smash{\overline{\R}}^{N\by n^2}_s\big)\big) \larrow \R
\end{equation}
is given by
\[
\langle \vartheta, \Phi \rangle\, :=\, \int_E \int_{\smash{\overline{\R}}^{N\by n^2}_s} \Phi(x,\X)\, d[\vartheta(x)] (\X)\, dx.
\]

\noi \textbf{Definition} (Young Measures). {\it The subset of the unit sphere of $L^\infty_{w^*}$ consisting of probability-valued maps is the set of Young measures from $E$ into the compactification $\smash{\overline{\R}}^{N\by n^2}_s$:}
\[
\mY\big(E,\smash{\overline{\R}}^{N\by n^2}_s\big)\, :=\, \Big\{ \vartheta\, \in \, L^\infty_{w^*}\big( E, \mM\big(\smash{\overline{\R}}^{N\by n^2}_s\big)\big)\, \Big| \, \vartheta(x) \in \mP\big(\smash{\overline{\R}}^{N\by n^2}_s\big),\text{ a.e. }x\in E\Big\}.
\]

We finally note the following well known facts (for their proofs see e.g.\ \cite{FG}): 
\smallskip

\noi (a) \emph{The set of Young measures above is convex and sequentially weakly* compact.}  

\noi (b) \emph{Every measurable map $v : E\sub \R^n \larrow \R^{N\by n^2}_s$ induces a Young measure $\de_v$ given by $\de_v(x):= \de_{v(x)}$.} 

\noi (c) \emph{Let $v^m,v^\infty : E\sub \R^n\larrow \R^{N\by n^2}_s$ be measurable, $m\in \N$. Up to the passage to subsequences, we have}
\[
v^m \larrow v^\infty \ \text{ a.e.\ on }E\ \ \ \ \Longleftrightarrow \ \ \ \ \de_{v^{m}} \weakstar \de_{v^\infty} \text{ in }\mY\big(E,\smash{\overline{\R}}^{N\by n^2}_s\big).
\]

\section{The prescribed singular value problem  
}  \label{section3}

In this section we prove Theorem \ref{mainAff} and Corollary \ref{cor:main}. To this aim we start by recalling for the convenience of the reader some well known results about generalised convex hulls of sets of matrices (for further details we refer to the books \cite{DM2} and \cite{D}). 

For $Q\in\mathbb{R}^{N\times n}$ we set
\[
T\left( Q \right) \, :=\, \Big(
Q,\, \mathrm{adj}_{2}Q,\, \ldots,\, ,\mathrm{adj}_{N\wedge n}Q\Big)\,
\in\mathbb{R}^{\tau(N,n)},
\]
where $\mathrm{adj}_{s}Q$ stands for the matrix of all $s\times
s$ subdeterminants of the matrix $Q$, $1\leq s\leq N\wedge
n=\min\left\{  N,n\right\}  $ and 

\[
\tau\left(  N,n\right)  \,:=\, \underset{s=1}{\overset{N\wedge
n}{\sum}}\binom {N}{s}\binom{n}{s}\ , \ \ \ \  \ \binom{N}{s}
\, =\, \frac{N!}{s!\left(  N-s\right)  !}.
\]

\begin{definition}\label{convfunctions}
Consider a function $f:\mathbb{R}^{N\times n}\rightarrow\mathbb{R}\cup\left\{  +\infty\right\}$.
\begin{enumerate}
\item $f$ is said to be \emph{polyconvex} if there exists a convex function
$g:\mathbb{R}^{\tau(N,n)}\to\mathbb{R}\cup\left\{
+\infty\right\}$ such that
$f(Q)=g(T(Q))$.
\item $f$ is said to be \emph{rank one convex} if
\[
f\big(  \lambda Q+(1-\lambda)R\big) \, \leq \, \lambda\,f\left(
Q\right) \,+\, \left(  1-\lambda\right) f\left( R\right)
\]
for every $\lambda\in\left[  0,1\right] $ and every
$Q,R \in\mathbb{R}^{N\times n}$ with
$\rk(R-Q)  =1$.
\end{enumerate}
\end{definition}

It is well known that if a function is polyconvex, then it is rank one convex.
Next we recall the corresponding notions of convexity for sets.

\begin{definition}\label{definitiongeneralizedconvexities}
Let $E$ be a subset of $\mathbb{R}^{N\times n}$.
\begin{enumerate}
\item
We say that $E$ is
\emph{polyconvex} if there exists a convex set $K\sub \mathbb{R}^{\tau(N,n)}$ such that
$\left\{Q\in\mathbb{R}^{N\times n}: T(Q)\in K\right\}=E.$

\item  We say that $E$ is
\emph{rank one convex} if for every $\lambda\in[0,1]$ and
for every $Q,R\in E$ such that $\rk(Q-R)=1$, then
$\lambda Q+(1-\lambda)R\in E.$
\end{enumerate}
\end{definition}

\begin{definition}\label{convexhulls}
The \emph{polyconvex} and \emph{rank one convex hulls} of a set
$E\sub \mathbb{R}^{N\times n}$ are, respectively, the smallest
polyconvex and rank one convex sets containing $E$ and are, respectively, denoted by
$\operatorname*{Pco}E$ and $\operatorname*{Rco}E$.
\end{definition}

Obviously one has the following inclusions:
$E\subseteq \operatorname*{Rco}E \subseteq \operatorname*{Pco}E\subseteq\operatorname*{co}E,$ where
$\operatorname*{co}E$ denotes the convex hull of $E$. Let us also recall the next characterisation of the rank one convex hull:
 \begin{equation}
 \label{Rico}
 \displaystyle{\operatorname*{Rco}\, E\, =\,
\bigcup_{i\in\mathbb{N}}\mathrm{R}_i\mathrm{co}\, E},
\end{equation} 
where $\mathrm{R}_0\mathrm{co}E=E$ and
$$
\mathrm{R}_{i+1}\mathrm{co}\, E\, =\, \left\{ Q \in\mathbb{R}^{N\times n}:\
\begin{array}
[c]{c}\vspace{0.2cm} Q=\lambda A+(1-\lambda) B,\ \lambda\in[0,1],\\
A,B\in \mathrm{R}_i\mathrm{co}\, E,\ \rk(A-B)\leq 1
\end{array}\right\},\ i\ge 0.
$$

It is well known that, for $E\sub \mathbb{R}^{N\times n}$, (cfr. \cite[Proposition 2.36]{D})
\begin{equation}
\label{coE}
\operatorname*{co}E \,=\, \Big\{Q\in\mathbb{R}^{N\times n}\,:\, f(Q) \leq0,\ \text{for every convex function }
f\in{\mathcal{F}}^{E}_\infty\Big\}
\end{equation}
where
$$
{\mathcal{F}}^{E}_\infty   \, =\, \Big\{  f:\mathbb{R}^{N\times n}
\rightarrow\mathbb{R}\cup\left\{ +\infty\right\}
\, : \,\left.  f\right\vert _{E}\leq0\Big\} \,.
$$

Analogous representations to $(\ref{coE})$ can be obtained in the
polyconvex and rank one convex cases:
\begin{align*}
&\operatorname*{Pco}E \, =\, \Big\{ Q \in\mathbb{R}^{N\times n}:
f(Q) \leq0,\ \text{for every polyconvex function }
f\in{\mathcal{F}}^{E}_\infty\Big\}, 
\\
&\operatorname*{Rco}E \,=\, \Big\{ Q \in\mathbb{R}^{N\times n}:
f(Q) \leq0,\ \text{for every rank one convex function }
f\in{\mathcal{F}}^{E}_\infty\Big\}.
\end{align*}

For a matrix $A\in \R^{N\times n}$, we denote by $\lambda_i(A)$, with $0\leq \lambda_i(A)\leq \lambda_{i+1}(A), i=1,\dots ,N\wedge n$, the singular values of $A$, that is, the eigenvalues of the matrix $(A^\top A)^{\frac 12}\in \R^{n\times n}$. The following singular value decomposition theorem can be  deduced from \cite[Theorem 2.6.3]{HJ}.
 
 \begin{theorem}\label{decomposition}
 Let $A\in \R^{N\times n}$. Then there exist $U\in \R^{N\times N}, V\in \R^{n\times n}$ unitary matrices and $D\in \R^{N\times n}$ rectangular diagonal matrix such that $A=UDV^t$.
 \end{theorem}
 
Using the singular values we can define the unitarily invariant norms on $\R^{N\times n}$ (known as Ky Fan k-norms (cf.\  \cite[Section 7.4.8]{HJ}) as
\[
\ \ \ \|A\|_k \,:=\,  \sum_{i=0}^{k-1}\lambda_{n-i}(A), \ \ \ k \in \{ 1,2,\dots, n \wedge N\}.
\]
Choosing $k=1$ we obtain the useful property of the maximum singular value $\lambda_{n \wedge N}(A)$ of being a norm in $\R^{N\times n}$.

As done in \cite[Theorem 7.16]{DM1} for the case $n=N$, in the next theorem we characterise the rank-one convex hull of the set 
\begin{equation}\label{definition_E}
E \, =\, \Big\{Q \in \R^{N\times n}\, :\ \lambda_i(Q)=1,\, i=1, \dots, N\wedge n \Big\}.
\end{equation}
It turns out that, in this particular case where the singular values are all equal, $\Rco E$ is indeed a convex set, as $\Co E=\Rco E=\Pco E$.

\begin{theorem}\label{th:rco}
Let $E$ be defined by (\ref{definition_E}). Then 
$\Co E=\Pco E=\Rco E$, 
\[
\Rco E \,= \, \Big\{Q \in \R^{N\times n}: \lambda_{N\wedge n}(Q)\leq 1 \Big\}
\]
and
\[
	\Int\Rco E \,=\, \Big\{Q \in \R^{N\times n}: \lambda_{N\wedge n}(Q) < 1 \Big\}.
\]

\end{theorem}

\begin{proof}
Let $X :=\{Q \in \R^{N\times n}: \lambda_n(Q)\leq 1\}$. The claim will follow from the inclusions  $\Co E\sub X$ and $X\sub \Rco E$. The first one is a direct consequence of the convexity of the function $A\mapsto \lambda_n(A)$ (that is convex being a norm) and of the characterisation of the convex envelope of a set given by \eqref{coE}. 
To prove the second inclusion, by the decomposition  Theorem \ref{decomposition}, it is enough to prove that a general rectangular diagonal matrix $D\in X$ is in the rank-one convex envelope of the subset of rectangular diagonal matrices in $E$. The proof then is identical to that of \cite[Theorem 7.17]{D} where we have to consider the notation $\text{diag}(a_1,\dots,a_n)$ as a rectangular diagonal matrix made by a square diagonal block with entries $(a_1,\dots,a_n)$ and a $(N-n)\times n$ or $(n-N)\times N$ zero block accordingly with the case $N>n$ or $n>N$. 

The characterisation of the interior follows from the continuity of $A\mapsto \lambda_n(A)$  ($\lambda_n$  is a norm) and from the fact that if $Q \in \Int\Rco E$ then $\lambda_n(Q)\neq 1$. The last claim can be easily proved by contradiction,  assuming without loss of generality that $Q$ is diagonal and  perturbing it by considering $Q+\varepsilon Z$ where $Z\in \R^{N\times n}$ is  such that $(Z)_{n\wedge N,n\wedge N}=1$ and the other entries of $Z$ are null. Then $Q+\varepsilon Z$ would not belong to $\Rco E$, which is a contradiction. 
 \end{proof}

To prove Theorem \ref{mainAff} we can apply the general existence theorem of  \cite{DP} (note that similar existence theorems have been obtaind by Kirchheim in \cite{Ki1, Ki2} and by M{\"u}ller and {\v{S}}ver{\'a}k in \cite{MS1, MS2}). To this aim we need the notion of Approximation Property defined below.

\begin{definition}
[Approximation Property]\label{Approximation property}Let $E\sub K\left(
E\right)  \sub \mathbb{R}^{N\times n}.$ The sets $E$ and $K\left(  E\right)
$ are said to have the\textit{\ }\emph{approximation property} if there exists
a family of closed sets $E_{\delta}$ and $K\left(  E_{\delta}\right)  $,
$\delta>0$, such that

(1) $E_{\delta}\sub K\left(  E_{\delta}\right)  \sub\operatorname*{int}%
K\left(  E\right)  $ for every $\delta>0;$

(2) for every $\epsilon>0$ there exists $\delta_{0}=\delta_{0}\left(
\epsilon\right)  >0$ such that $\operatorname*{dist}(Q,E)\leq\epsilon$ for
every $Q\in E_{\delta}$ and $\delta\in\left[  0,\delta_{0}\right]  $;

(3) if $Q\in\operatorname*{int}K\left(  E\right)  $ then $Q\in K\left(
E_{\delta}\right)  $ for every $\delta>0$ sufficiently small.
\end{definition}

We therefore have the following result (cf. \cite[Theorem 7 and 9]{DP} and also \cite[Section 10.2.1]{D}).

\begin{theorem}
\label{Relaxation property theorem for RcoE}Let $E\sub\mathbb{R}^{N\times
n}$ be compact and suppose that there exist $E_{\delta}$ and $K\left(  E_{\delta}\right)  =\operatorname*{Rco}E_{\delta}$  satisfying the  approximation property with $K= \Int\Rco E$. Let $\Omega\sub\mathbb{R}^{n}$ be open and $g\in \text{\emph{Aff}}_{\text{\emph{piec}}}\left(  \overline{\Omega}, \mathbb{R}^{N}\right)$ be
such that
\[
\ \ \ \ \ \D g\left(  x\right) \,  \in \, E\, \cup\, \Int\Rco E \text{, \ a.e. }x\in \Omega.
\]
Then there exists (a dense set of) $u\in W_g^{1,\infty}\left(
\Omega,\mathbb{R}^{N}\right)  $ such that
\[
\D u\left(  x\right)  \in E\text{, \ a.e. }x\in \Omega.
\]

\end{theorem}

\begin{remark}\label{rem:int}
	If the set $K$ is open, $g$ can be taken in
$C_{\text{\emph{piec}}}^{1}\left(  \overline{\Omega}, \mathbb{R}^{N}\right)  $ (cf. Corollary
10.15 or Theorem 10.16 in \cite{D}), with $\D g\left(
x\right)  \in E\cup K$. While if $K$ is open and convex, $g$ can be
taken in\ $W^{1,\infty}\left(  \Omega,\mathbb{R}^{N}\right)  $ provided
\[
\D g\left(  x\right)  \in C\text{, \ a.e. }x\in \Omega,
\]
where $C\sub K$ is compact (cf. Corollary 10.21 in \cite{DM1}).
\end{remark}

\noi \textbf{Proof of Theorem \ref{mainAff} and Corollary \ref{cor:main}.} We prove the approximation property with
\[
E_{\delta} \,=\, \Big\{Q \in \R^{N\times n}\, :\, \lambda_i(Q)=1-\delta, \, i=1,\dots, n\wedge N
\Big\}.
\]
Indeed it is easy to check that
\[
\begin{split}
\Rco E_{\delta}\, &=\, \Big\{Q \in \R^{N\times n}\, :\, \lambda_{n\wedge N}(Q)\leq 1-\delta 
\Big\} 
\\ 
& \sub \, \Big\{Q \in \R^{N\times n}\,:\, \lambda_{n\wedge N}(Q)< 1 \Big\}
\\
& =\, \Int \Rco E\,.
\end{split}
\]
To check the  second condition we observe that for any fixed $Q \in E_{\delta}$,
then up to a multiplication by unitary matrices (see Theorem \ref{decomposition}), we can assume that $Q=(1-\delta)\tilde{\mathrm{I}}$, where $\tilde{\mathrm{I}}$ is the rectangular identity  matrix belonging to $E$. Therefore $\|Q-\tilde{\mathrm{I}}\|\leq c\delta$. The third condition can be easily verified arguing as in the last part of the proof of Theorem \ref{th:rco}. We can therefore prove Theorem \ref{mainAff} by applying Theorem \ref{Relaxation property theorem for RcoE}. The claim of Corollary \ref{cor:main} can be proved by arguing as before and appealing to Remark \ref{rem:int}.
\qed

\section{$\mD$-solutions to the PDE system arising in vectorial Calculus of Variations in $L^\infty$} \label{section4}

In this section we establish the proofs of Theorem \ref{theorem3} and of Proposition \ref{corollary4} by utilising Corollary \ref{cor:main} which was proved in Section \ref{section3}. 

\BPP \ref{corollary4}. We begin by showing the next consequence of the solvability of the singular value problem.

\begin{claim} \label{claim} Given an open set $\Om\sub \R^n$ and $g\in W^{1,\infty}(\Om,\R^N)$ with $N\leq n$, for any constant $c>\|\D g\|_{L^\infty(\Om)}$ there exist (an infinite set of) maps $u : \Om \sub \R^n \larrow \R^N$ in $ W_g^{1,\infty}(\Om,\R^N)$ such that\footnote{One of the referees of this paper brought to our attention that the tensor $\D u\, \D u^{\!\top}$ is the so called Left-Cauchy-Green Tensor which appears in nonlinear elasticity.}
\begin{equation}\label{3.1}
\left\{
\begin{array}{rl}
\D u\, \D u^{\!\top}\, =\, c^2\mathrm{I},  & \text{ a.e.\ on }\Om,\\
u\,=\,g, \ \ & \text{ on }\p\Om.
 \end{array}
 \right.
\end{equation}
\end{claim}

\BPC \ref{claim}. Since the matrix-valued maps $(\D u\, \D u^{\!\top})^{1/2}$ and $(\D u^{\!\top}\D u)^{1/2}$ have the same non-zero eigenvalues, by the results of the previous section the Dirichlet problem
\begin{equation}\label{3.2}
\left\{
\begin{array}{rl}
\si_\al\big((\D v\, \D v^{\!\top})^{1/2}\big)\, =\, 1, \ \ \ & \text{ a.e.\ on }\Om, \ \al=1,...,N, \ms
\\
v\,=\,g/c, & \text{ on }\p\Om,
 \end{array}
 \right.
\end{equation}
(where $\{\si_1,...,\si_N\}$ symbolise the eigenvalues of the $N\by N$ matrix in increasing order) has solutions $v \in W_{g/c}^{1,\infty}(\Om,\R^N)$ because a.e.\ on $\Om$ we have that
\[
\begin{split}
\max_{\al=1,...,N}\si_\al\Big( \big(\D (g/c)\, \D (g/c)^{\!\top}\big)^{1/2}\Big)\, &=\, \frac{1}{c}\Big\{\si_N\big(\D g\, \D g^{\!\top}\big)\Big\}^{1/2}
\\
&=\, \frac{1}{c}\Big\{\max_{|e|=1} \big(\D g\, \D g^{\!\top}\big) : (e\ot e) \Big\}^{1/2}
\\ 
&=\, \frac{1}{c}\max_{|e|=1} \big|e^\top \D g\big|
\\
&\leq\, \frac{1}{c}\big\| \D g\big\|_{L^\infty(\Om)}
\\
&\leq\, 1-\de,
\end{split}
\]
for some $\de>0$. This is a consequence of Corollary \ref{cor:main}. By rescaling $u:=cv$, we have the existence of an infinite set of solutions $u \in W_{g}^{1,\infty}(\Om,\R^N)$ of
\begin{equation}\label{3.3}
\left\{
\begin{array}{rl}
\si_\al\big((\D u\, \D u^{\!\top})^{1/2}\big)\, =\, c, \ & \text{ a.e.\ on }\Om,\ \al=1,...,N,
\ms\\
u\,=\,g, & \text{ on }\p\Om.
 \end{array}
 \right.
\end{equation}
Next, by the Spectral Theorem, for each such $u$ there is a measurable map with values in the orthogonal matrices
\[
U\ : \ \ \Om\sub \R^n \larrow O(N,\R) \sub \R^{N\by N}
\]
such that
\[
\ \ \ \D u\, \D u^{\!\top}\, =\, U 
\left[
\begin{array}{lr}
\si_1\big((\D u\, \D u^{\!\top})^{1/2}\big)^2 &    \mathrm{O}\
\\
\hspace{55pt}  \ddots &
\\
\ \mathrm{O}    & \!\!\!\!\!\!\!\!\!\!\!\!\!\!\!\si_N\big((\D u\, \D u^{\!\top})^{1/2}\big)^2
\end{array}
\right]
U^\top\, =\, U(c^2\mathrm{I})U^\top\, =\, c^2\mathrm{I},
\]
a.e.\ on $\Om$. The claim thus ensues.   \qed
\ms

Now we complete the proof of the proposition. By our assumptions on $\H$ (see Theorem \ref{theorem3}), for any $u$ as above we have
\[
\H(\D u)\, =\, h\big(\D u\, \D u^{\!\top}\big)\, =\, h\big(c^2\mathrm{I}\big)
\]
and by splitting the identity as
\[
c^2\mathrm{I}\, =\, [c \mathrm{I}\, |\, \mathrm{O}] \cdot [c \mathrm{I}\, |\, \mathrm{O}]^\top
\]
where $[c \mathrm{I}\, |\, \mathrm{O}] \in \R^{N\by N} \by \R^{N\by (n-N)}$, we have
\[
\ \ \ \H(\D u)\, =\, h\Big( [c \mathrm{I}\, |\, \mathrm{O}] \cdot [c \mathrm{I}\, |\, \mathrm{O}]^\top\Big)\, =\, \H\big([c \mathrm{I}\, |\, \mathrm{O}]\big), \ \  \ \text{ a.e.\ on }\Om.
\]
Further,
\[
\det\big(\D u\, \D u^{\!\top}\big)\, =\, \det\left(
\big(\D u\, \D u^{\!\top}\big)^{1/2}\right)^{\!2}\, =\, \left\{\prod_{\al=1}^N\si_\al\! \left(
\big(\D u\, \D u^{\!\top}\big)^{1/2}\right) \! \right\}^{\!2}=\,c^{2N},
\]
a.e.\ on $\Om$. Note now that in view of the symmetry of the derivative of $h$, we have
\[
\begin{split}
\H_{P_{\al i}}(P)\, &=\, \big(h\big(PP^\top\big)\big)_{P_{\al i}}
\\
&=\, \sum_{\be,\ga=1}^N \sum_{j=1}^n h_{X_{\be \ga}}\big(PP^\top\big)\Big(P_{\be j}\de_{\al \ga}\de_{ij}\, +\, P_{\ga i}\de_{\al \be}\de_{ij} \Big)
\\
& =\, 2 \sum_{\be=1}^N  h_{X_{\al \be}}\big(PP^\top\big)\, P_{\be i},
\end{split}
\]
for all $\al=1,...,N$ and $i=1,...,n$. Hence, by the last two equalities and our assumption on $\H$ we have that $\H_P(\D u)$ has full rank:
\[
\rk\left(\H_P(\D u)\right)\, =\, \rk\left(h_X\big(\D u\, \D u^{\!\top}\big)\, \D u\right)\, =\, \rk(\D u)\, =\, N,
\]
a.e.\ on $\Om$. By \eqref{1.2}, this implies 
\[
\left[\H_P(\D u)\right]^\bot=\, \textrm{Proj}_{R(\H_P(\D u))^\bot}\, =\, \textrm{Proj}_{(\R^N)^\bot}\, =\, \textrm{Proj}_{\{0\}}\, =\, 0, 
\]
a.e.\ on $\Om$, as desired. The proposition ensues. \qed
\ms

\BPT \ref{theorem3}. Given $\Om\sub \R^n$ open, $g\in W^{1,\infty}(\Om,\R^N)$ and $c>\|\D g\|_{L^\infty(\Om)}$, let $u\in \mA_c$ be any of the solutions obtained in Proposition \ref{corollary4} of the systems of fully nonlinear equations \eqref{1.11}. We set
\begin{equation}\label{3.4}
\A^\infty_{\al i \be j}(P)\, :=\, \H_{P_{\al i}}(P)\, \H_{P_{\be j}}(P)\,+ \, \H(P)\sum_{\ga=1}^N \big[\H_P(P)\big]^\bot_{\al \ga}\H_{P_{\ga i}P_{\be j}}(P), 
\end{equation}
where $\al,\be=1,...,N$ and $i,j=1,...,n$. In order to show that $u$ is a $\mD$-solution of \eqref{1.1}, we must show that for any diffuse hessian of $u$
\begin{equation}\label{3.5}
\ \ \ \ \de_{\D^{1,h_{m}}\D u}\weakstar \mD^2 u \  \ \text{ in }\mY\big(\Om,\smash{\overline{\R}}^{N\by n^2}_s\big),
\end{equation}
as $m\ri \infty$, we have
\begin{equation}\label{3.6}
\ \ \ \int_{\smash{\overline{\R}}^{N\by n^2}_s} \Phi(\X)\, \big[ \A^\infty\big(\D u(x)\big):\X\big]\, d\big[\mD^2u(x)\big](\X)\, =\, 0
\end{equation}
for a.e.\ $x\in \Om$ and any fixed $\Phi \in C_c\big( {\R}^{N\by n^2}_s \big)$. By \eqref{1.11}, we have $\H\big(\D u(x)\big)=\, $const for a.e.\ $x\in \Om$. Fix such an $x\in \Om$, $0<|h|<\dist(x,\p\Om)$ and $i\in \{1,...,n\}$. Then, by Taylor's theorem, we have
\[
\begin{split} 
 0\, &=\, \H\big(\D u(x+he^i)\big) - \H\big(\D u(x)\big) \, =
\\
&=\,\sum_{\be=1}^N\sum_{j=1}^n\int_0^1 \H_{P_{\be j}}\Big(\D u(x) + \la \Big[\D u(x+he^i)-\D u(x)\Big] \Big)\,d\la\centerdot
\\ 
&\hspace{55pt} \centerdot \Big[\D_j u_\be(x+he^i)-\D_j u_\be (x)\Big].
\end{split}
\]
This implies the identity
\begin{equation}\label{3.7}
\begin{split} 
   &\sum_{\be=1}^N\sum_{j=1}^n \H_{P_{\be j}}\big(\D u(x)\big)\, \big(D^{1,h}_i\D_ju_\be\big)(x) 
\\
 &+\, \sum_{\be=1}^N\sum_{j=1}^n \int_0^1 \bigg[\H_{P_{\be j}}\Big(\D u(x)  + \la \Big[\D u(x+he^i)-\D u(x)\Big] \Big)
\\
& - \H_{P_{\be j}}\big(\D u(x)\big)\bigg]\,d\la \, \big(D^{1,h}_i\D_ju_\be\big)(x) \, =\, 0.
\end{split}
\end{equation}
Let now $h_m\ri 0$ as $m\ri \infty$ be an infinitesimal sequence giving rise to a diffuse hessian as in \eqref{3.5}. We define the ``error" tensor 
\begin{equation}\label{3.8}
\begin{split}
\textbf{E}^m_{\al i \be j}(x)\,& :=\, \H_{P_{\al i}}\big(\D u(x)\big)\int_0^1 \bigg[\H_{P_{\be j}}\Big(\D u(x)  + \la \Big[\D u(x+h_m e^i)
\\
&\ \ \ \ \ -\D u(x)\Big] \Big) - \H_{P_{\be j}}\big(\D u(x)\big)\bigg]\,d\la  ,
 \end{split}
\end{equation}
for a.e.\ $x\in \Om$, where $\al,\be=1,...,N$, $i,j=1,...,n$ and $m\in \N$. Then, in view of \eqref{3.8} and of the following consequence of
 \eqref{1.11}
\[
\sum_{\ga=1}^N\H\big(\D u(x)\big)\left[\H_P\big(\D u(x)\big)\right]^\bot_{\al \ga}\H_{P_{\ga i}P_{\be j}}\big(\D u(x)\big)\, =\, 0,
\]
 identity \eqref{3.7} (for $h=h_m$) yields
\begin{equation} \label{3.9a}
\begin{split}
& \sum_{\be,\ga=1}^N\sum_{i,j=1}^n \bigg\{\bigg(\H_{P_{\al i}}(\D u)\, \H_{P_{\be j}}(\D u) 
\\
& +\, \H(\D u)\big[\H_P(\D u)\big]_{\al \ga}^\bot \H_{P_{\ga i}P_{\be j}}(\D u)\bigg) \, +\, \textbf{E}^m_{\al i \be j}\bigg\}\D^{1,h_m}_{i}\D_j u_\be\, =\, 0,
\end{split}
\end{equation}
a.e.\ on $\Om$ and for all $\al=1,...,N$. In view of definition \eqref{3.4}, we rewrite \eqref{3.9a} compactly as
\[
\Big(\A^\infty\big(\D u(x)\big)+ \textbf{E}^m(x)\Big):\D^{1,h_m}\D u(x)\, =\, 0, 
\]
for a.e.\ $x\in \Om$. By multiplying by $\Phi(\D^{1,h_m}\D u)$, this gives
\begin{equation}  \label{3.9}
\int_{\smash{\overline{\R}}^{N\by n^2}_s} \Big|\Phi(\X)\, \Big(\A^\infty\big(\D u(x)\big)+ \textbf{E}^m(x)\Big):\X \Big|\, d\big[\de_{\D^{1,h_m}\D u}(x)\big|(\X)\, =\, 0,
\end{equation}
for a.e.\ $x\in \Om$. Since $\D u \in L^{\infty}(\Om,\R^{N\by n})$, by the continuity of the translations in $L^1$ we have 
\[
\D u(\cdot +he^i) \larrow \D u, \ \ \ i=1,...,n,
\]
as $h\ri 0$, in $L^1_{\text{loc}}(\Om,\R^{N\by n})$ and hence along, if necessary, a further subsequence $(m_k)_1^\infty$ we have 
\[
\D u\big(x+h_{m_k}e^i\big) \larrow \D u(x), \ \ \ \text{for a.e. }x\in \Om,
\]
as $k\ri \infty$. Since $\H \in C^1(\R^{N\by n})$ and $\D u \in L^{\infty}(\Om,\R^{N\by n})$, the Dominated Convergence Theorem and \eqref{3.8} imply that 
\begin{equation}\label{3.10}
|\textbf{E}^{m_k}| \larrow 0, \ \ \text{ in }L^1_{\text{loc}}(\Om),
\end{equation}
as $k\ri \infty$. Let $E\sub \Om$ be a compact set. Then, \eqref{3.9} gives
\begin{equation}  \label{3.11}
\int_E\int_{\smash{\overline{\R}}^{N\by n^2}_s} \Big|\Phi(\X)\,\Big(\A^\infty\big(\D u(x)\big)+ \textbf{E}^{m_k}(x)\Big):\X \Big|\, d\big[\de_{\D^{1,h_{m_k}}\D u}(x)\big](\X)\,dx\, =\, 0.
\end{equation}
We define the Carath\'eodory maps
\[
\begin{split}
\ \ \ \ \ \  \Psi^m(x,\X)\, &:=\, \chi_E(x)\Big|\Phi(\X)\,\Big(\A^\infty\big(\D u(x)\big)+ \textbf{E}^m(x)\Big):\X \Big|, \ \ \ m\in\N,
\\
\Psi^\infty(x,\X)\, &:=\, \chi_E(x)\Big|\Phi(\X)\,\Big(\A^\infty\big(\D u(x)\big)\Big):\X \Big|,
\end{split}
\]
which are elements of the Banach space $L^1\big(E,C\big(\smash{\overline{\R}}^{N\by n^2}_s \big)\big)$ by the compactness of the support of $\Phi$ and of the set $E$. We also have that 
\begin{equation}\label{3.12}
\ \ \ \Psi^{m_k} \larrow \Psi^\infty, \ \text{ as }k \ri \infty \ \text{ in }L^1\big(E,C\big(\smash{\overline{\R}}^{N\by n^2}_s \big)\big),
\end{equation}
because of the estimate
\[
\|\Psi^{m_k} -\Psi^\infty\|_{L^1(E,C(\smash{\overline{\R}}^{N\by n^2}_s ))}\, \leq\, \max_{\X \in \supp(\Phi)}\Big|\Phi(\X)\X\Big| \int_E \big|\textbf{E}^{m_k}(x)\big|\,dx
\]
which yields that the right hand side vanishes as $k\ri\infty$ as a consequence of\eqref{3.10}. The weak*-strong continuity of the duality pairing \eqref{dp}, \eqref{3.12} and \eqref{3.5} allow us to pass to the limit as $k\ri \infty$ in \eqref{3.11} and deduce
\[
\int_E\int_{\smash{\overline{\R}}^{N\by n^2}_s} \Phi(\X)\,\Big[\A^\infty\big(\D u(x)\big) :\X \Big] \, d\big[\mD^2 u(x)\big](\X)\,dx\, =\, 0.
\]
Since $E\sub \Om$ is an arbitrary compact set, the theorem ensues. \qed
\ms

\section{Remarks and open questions}

In this final section we make some further comments regarding the results stated in Theorem \ref{theorem3}, pointing also towards certain relevant open questions which we have not answered in this work. 
\smallskip

\noi $\bullet$  {An interesting question which we not attempt to answer herein regards the relation of $\mD$- and viscosity solutions in the {scalar} case of $N=1$, in particular in relation to the first order vectorial equation $\mathrm{H}(\D u)=c$. The content of Theorem \ref{theorem3} is that {\it all $W^{1,\infty}$ strong solutions to the first order system \eqref{1.11} are $\mD$-solutions to \eqref{1.1}}. This is a kind of generalisation of the {\it scalar} implication that ``differentiable solutions to the Eikonal equation $|\D u|=c$ are $\infty$-Harmonic". However, we refrained from mentioning this explicitly earlier since, when restricted to the scalar case, the notion of $\mD$-solutions is generally {weaker} than that of viscosity solutions. 
	In the recent paper \cite{K10} some further insights regarding the behaviour of $C^1$ $\mD$-solutions and their relation to the variational problem are presented, even though they do not answer this particular question.}

\smallskip

\noi $\bullet$  {A further interesting open question regards whether $C^1$-regular $\mD$-solutions to the system are in some sense ``critical points" to the functional. Clearly, $C^1$ is a very optimistic regularity expectation for putative minima or extrema of any kind and one might hope that this prominent class would have the most favourable properties. At present we do not have a definitive answer, primarily because there is no clear way of how to define an adequate notion of ``critical points" for supremal functionals.
}

\smallskip

\noi $\bullet$  {We close by noting that one might be tempted to call the solutions we construct ``critical points" of the functional, in the sense that they solve the equilibrium equations whilst not minimising the functional. However, the reader should note that our solutions are {not critical points in the classical sense} of Calculus of Variations for integral functionals.  The situation for supremal functionals, even in the scalar case, is indeed far from being well understood. In particular, for supremal functionals, even global minimisers may not solve the equations! To the best of our knowledge, nowhere in the literature any kind of $L^\infty$-notion of critical points has been defined. The only existing variational notions are those of absolute minimisers and their vectorial generalisations ($\infty$-minimal maps in \cite{K2} and tight maps in \cite{SS}).}

 \ms

\noi \textbf{Acknowledgement.} The authors are indebted to the referees for the careful reading of the paper and for their suggestions which improved both the content and the presentation, in particular for their perspicacious comments and their clarifying questions relating to Section 5. N.K. would like to thank Bernard Dacorogna for the selfless share of expertise on the problem of prescribing the singular values and Jan Kristensen for their inspiring scientific discussions about $L^\infty$ variational problems. G.P. is a member of the Gruppo Nazionale per l'Analisi Matematica, la Probabilit\`a e le loro Applicazioni (GNAMPA) of the Istituto Nazionale di Alta Matematica (INdAM).

\ms
\ms

\bibliographystyle{amsplain}

\begin{thebibliography}{30}

\bibitem[AK]{AK} H. Abugirda, N. Katzourakis, \emph{Existence of $1D$ vectorial Absolute Minimisers in $L^\infty$ under minimal assumptions}, Proceedings of the AMS, 145 (6), 2567 - 2575 (2017), DOI: https://doi.org/10.1090/proc/13421.

\bibitem[AM]{AM} L. Ambrosio, J. Mal\'y, \emph{Very weak notions of differentiability}, Proceedings of the Royal Society of Edinburgh A 137 (2007), 447 - 455.

\bibitem[A1]{A1} G. Aronsson, \emph{Minimization problems for the functional $sup_x \mF(x,
f(x), f'(x))$}, Arkiv f\"ur Mat. 6 (1965), 33 - 53.

\bibitem[A2]{A2} G. Aronsson, \emph{Minimization problems for the functional $sup_x \mF(x,
f(x), f'(x))$ II}, Arkiv f\"ur Mat. 6 (1966), 409 - 431.

\bibitem[A3]{A3} G. Aronsson, \emph{Extension of functions satisfying Lipschitz conditions}, Arkiv f\"ur Mat. 6 (1967), 551 - 561.

\bibitem[A4]{A4} G. Aronsson, \emph{On the partial differential equation $u_x^2 u_{xx} + 2u_x u_y u_{xy} + u_y^2 u_{yy} = 0$}, Arkiv f\"ur Mat. 7
(1968), 395 - 425.

\bibitem[A5]{A5} G. Aronsson, \emph{Minimization problems for the functional $sup_x \mF(x, f(x), f'(x))$ III}, Arkiv f\"ur Mat. (1969), 509 - 512.

\bibitem[A6]{A6} G. Aronsson, \emph{On Certain Singular Solutions of the Partial
Differential Equation $u_x^2 u_{xx} + 2u_x u_y u_{xy} + u_y^2 u_{yy} =
0$}, Manuscripta Math. 47 (1984), 133 - 151.

\bibitem[A7]{A7} G. Aronsson, \emph{Construction of Singular Solutions to the $p$-Harmonic Equation and its Limit Equation for $p=\infty$}, Manuscripta Math.
56 (1986), 135 - 158.

\bibitem[ACJ]{ACJ} G. Aronsson, M. Crandall, P. Juutinen \emph{A tour of the theory of absolutely minimizing functions}, Bulletin of the AMS, New Series 41, 439--505 (2004).

\bibitem[BEJ]{BEJ} E. N. Barron, L. C. Evans, R. Jensen, \emph{The Infinity Laplacian, Aronsson's Equation and their Generalizations}, Trans. Amer. Math. Soc. 360, 77--101 (2008).

\bibitem[BJW1]{BJW1} E. N. Barron, R. Jensen and C. Wang, \emph{The Euler equation and absolute minimizers of $L^{\infty}$ functionals}, Arch. Rational Mech. Analysis 157 (2001), 255 - 283.

\bibitem[BJW2]{BJW2} E. N. Barron, R. Jensen, C. Wang, \emph{Lower Semicontinuity of $L^{\infty}$ Functionals} Ann. I. H. Poincar\'e AN 18 (2001) 495 - 517.


\bibitem[BCR]{BCR}A.C. Barroso, G. Croce, A. Ribeiro, \emph{Sufficient conditions for existence of solutions to vectorial differential inclusions and applications}, Houston J. Math. 39  (2013), 929 - 967.

\bibitem[BL]{BL} C. Le Bris, P. L. Lions, \emph{Renormalized solutions of some transport equations with partially $W^{1,1}$ velocities and applications},  Ann. di Mat. Pura ed Appl. 183 (2004) 97 - 130. 

\bibitem[CR]{CR} L. Capogna, A. Raich, \emph{An Aronsson type approach to extremal quasiconformal mappings}, Journal of Differential Equations Volume 253, 3 (2012), 851 - 877.

\bibitem[CFV]{CFV} C. Castaing, P. R. de Fitte, M. Valadier, \emph{Young Measures on Topological spaces with Applications in Control Theory and Probability Theory}, Mathematics and Its Applications, Kluwer Academic Publishers, 2004.

\bibitem[C]{C} M. G. Crandall, \emph{A visit with the $\infty$-Laplacian}, in \emph{Calculus of Variations and Non-Linear Partial Differential Equations}, Springer Lecture notes 1927, CIME, Cetraro Italy 2005.

\bibitem[Cr]{Cr} G. Croce, \emph{A differential inclusion: the case of an isotropic set}, ESAIM Control Optim. Calc. Var. 11 (2005), 122 - 138 (electronic).

\bibitem[D]{D} B. Dacorogna,  \emph{Direct Methods in the Calculus of Variations}, $2$nd Edition, Volume 78, Applied Mathematical Sciences, Springer, 2008.


\bibitem[DM1]{DM1} B. Dacorogna,  P. Marcellini, \emph{Cauchy-Dirichlet problem for first order nonlinear systems},  Journal of Functional Analysis 152 (1998), 404 - 446.

\bibitem[DM2]{DM2} B. Dacorogna,  P. Marcellini, \emph{Implicit Partial Differential Equations}, Progress in Nonlinear Differential Equations and Their Applications, Birkh\"auser, 1999.

\bibitem[DP]{DP} B. Dacorogna, G. Pisante, \emph{A general existence theorem for differential inclusions in the vector valued case}, Portugaliae Mathematica 62 (2005), 421 - 436.

\bibitem[DPR]{DPR} B. Dacorogna, G. Pisante, A. M. Ribeiro \emph{On non quasiconvex problems of the calculus of variations}, Discrete Contin. Dyn. Syst..13 (2005) 961 - 983.

\bibitem[DR]{DR} B. Dacorogna, A. M. Ribeiro, \emph{Existence of solutions for some implicit partial differential equations and applications to variational integrals involving quasi-affine functions},  Proc. Roy. Soc. Edinburgh Sect. A 134  (2004), 907 - 921.

\bibitem[DT]{DT} B. Dacorogna, C. Tanteri, \emph{On the different convex hulls of sets involving singular values},
Proc. Roy. Soc. Edinburgh Sect. A 128 (1998), 1261 - 1280.

\bibitem[Ed]{Ed} R.E. Edwards, \emph{Functional Analysis: Theory and Applications}, Dover, 2011.

\bibitem[E]{E} L.C. Evans, \emph{Partial Differential Equations}, AMS Graduate Studies in Mathematics 19.1, 2nd edition 2010.

\bibitem[EG]{EG} L.C. Evans, R. Gariepy, \emph{Measure theory and fine properties of functions}, Studies in advanced mathematics, CRC press, 1992.

\bibitem[FG]{FG} L.C. Florescu, C. Godet-Thobie, \emph{Young measures and compactness in metric spaces}, De Gruyter, 2012.

\bibitem[FL]{FL} I. Fonseca, G. Leoni, \emph{Modern methods in the Calculus of Variations: $L^p$ spaces}, Springer Monographs in Mathematics, 2007.

\bibitem[HJ]{HJ} R.A. Horn, Ch.R. Johnson, \emph{Matrix Analysis}, Cambridge University Press, 2012.

\bibitem[K1]{K1} N. Katzourakis,  \emph{$L^\infty$-Variational Problems for Maps and the Aronsson PDE system}, J.\ Differential Equations 253 (2012), 2123 - 2139.

\bibitem[K2]{K2} N. Katzourakis, \emph{$\infty$-Minimal Submanifolds}, Proceedings of the AMS 142 (2014), 2797 - 2811.


\bibitem[K3]{K3}N. Katzourakis,  \emph{On the Structure of $\infty$-Harmonic Maps}, Communications in PDE 39 (2014), 2091 - 2124.

\bibitem[K4]{K4} N. Katzourakis, \emph{Explicit $2D$ $\infty$-Harmonic Maps whose Interfaces have Junctions and Corners}, Comptes Rendus Acad. Sci. Paris, Ser.I  351 (2013), 677 - 680.

\bibitem[K5]{K5} N. Katzourakis,  \emph{Optimal $\infty$-Quasiconformal Immersions}, ESAIM Control, Opt. and Calc. Var., Vol. 21, Number 2 (2015), 561 - 582. 


\bibitem[K6]{K6} N. Katzourakis, \emph{Nonuniqueness in Vector-valued Calculus of Variations in $L^\infty$ and some Linear Elliptic Systems}, Comm. on Pure and Appl. Anal. 14 (2015), 313 - 327. 

\bibitem[K7]{K7} N. Katzourakis, \emph{An Introduction to viscosity Solutions for Fully Nonlinear PDE with Applications to Calculus of Variations in $L^\infty$}, Springer Briefs in Mathematics, 2015, DOI 10.1007/978-3-319-12829-0.
 
\bibitem[K8]{K8} N. Katzourakis,  \emph{Generalised solutions for fully nonlinear PDE systems and existence-uniqueness theorems}, Journal of Differential Equations 23 (2017), 641 - 686, DOI: 10.1016/j.jde. 2017.02.048.

\bibitem[K9]{K9} N. Katzourakis,  \emph{Absolutely minimising generalised solutions to the equations of vectorial Calculus of Variations in $L^\infty$}, Calculus of Variations and PDE 56 (1), 1 - 25 (2017) (DOI: 10.1007/s00526-016-1099-z).

\bibitem[K10]{K10} N. Katzourakis,  \emph{A new characterisation of $\infty$-Harmonic and $p$-Harmonic maps via affine variations in $L^\infty$}, Electronic Journal of Differential Equations, Vol. 2017 (2017),  No. 29, 1 - 19.


\bibitem[K11]{K11} N. Katzourakis,  \emph{Solutions of vectorial Hamilton-Jacobi equations are rank-one Absolute Minimisers in $L^\infty$}, Advances in Nonlinear Analysis, in press.

\bibitem[K12]{K12} N. Katzourakis,  \emph{Weak versus $\mD$-solutions to linear hyperbolic first order systems with constant coefficients}, ArXiv preprint, \url{http://arxiv.org/pdf/1507.03042.pdf}. 


\bibitem[KP]{KP} N. Katzourakis,  T. Pryer, \emph{On the numerical approximation of $\infty$-Harmonic mappings}, Nonlinear Differential Equations and Applications  23 (6), 1 - 23 (2016).

\bibitem[KM]{KM} \emph{Remarks on the Validity of the Maximum Principle for the $\infty$-Laplacian}, Le Matematiche, Vol. LXXI (2016) Ð Fasc. I, 63Ð74, DOI: 10.4418/2016.71.1.5.

\bibitem[Ki1]{Ki1} B. Kirchheim, \emph{Rigidity and geometry of microstructures}, Issue 16 of Lecture notes, Max-Planck-Institut fr Mathematik in den Naturwissenschaften Leipzig, 2003, 116 pages.

\bibitem[Ki2]{Ki2} B. Kirchheim, \emph{Deformations with finitely many gradients and stability of convex hulls}, Comptes Rendus de l'Acadmie des Sciences, Sries I, Mathematics, 332, 2001, 289-294.


\bibitem[MS1]{MS1} S. M{\"u}ller, V. {\v{S}}ver{\'a}k, \emph{Attainment results for the two-well problem by convex integration},  Geometric analysis and the calculus of variations, Internat. Press, Cambridge, MA (1996), 239 - 251.

\bibitem[MS2]{MS2} S. M{\"u}ller, V. {\v{S}}ver{\'a}k,
 \emph{Convex integration for Lipschitz mappings and counterexamples to regularity}, Ann. of Math., 157 (2003), 715 - 742.

\bibitem[P]{P} P. Pedregal, \emph{Parametrized Measures and Variational Principles}, Birkh\"auser, 1997.

\bibitem[Pi]{Pi} G. Pisante, \emph{Sufficient conditions for the existence of viscosity solutions for nonconvex Hamiltonians}, SIAM J. Math. Anal., 36(1):186â203, 2004.

\bibitem[SS]{SS} S. Sheffield, C.K. Smart, \emph{Vector-Valued Optimal Lipschitz Extensions}, Comm. Pure Appl. Math.  65, 128--154 (2012).

\bibitem[V]{V} M. Valadier, \emph{Young measures}, in ``Methods of nonconvex analysis", Lecture Notes in Mathematics 1446, 152-188 (1990).


\end{thebibliography}

\end{document}